\documentclass[12pt]{amsart}

\usepackage{lmodern}
\usepackage[fontsize=13pt]{scrextend}
\usepackage{fullpage}
\usepackage{mathtools}
\usepackage{amssymb,amsmath,amsthm,amscd,mathrsfs,graphicx}
\usepackage[dvipsnames]{xcolor}
\usepackage{bm}
\usepackage{hyperref}
\usepackage{dsfont}
\usepackage{enumerate}
\usepackage{epsfig}
\usepackage{float, graphicx}
\usepackage{latexsym, amsxtra}
\usepackage{pdfpages}
\usepackage{multicol}
\usepackage[normalem]{ulem}
\usepackage{psfrag}
\usepackage{stmaryrd}
\usepackage{tikz}
\usepackage[T1]{fontenc}
\usepackage{url}
\usepackage{verbatim}
\usepackage{xy}
\usepackage{indentfirst}
\usepackage{tikz-cd}
\usepackage{url}

\makeatletter
\def\thm@space@setup{%
	\thm@preskip=2ex \thm@postskip=2ex
}
\makeatother

\hypersetup{hidelinks}

\numberwithin{equation}{section}
\theoremstyle{plain}
\newtheorem{thm}{Theorem~}[section] 
\newtheorem{lem}[thm]{Lemma~}

\newtheorem{prop}[thm]{Proposition~}

\theoremstyle{remark}
\newtheorem{rmk}[thm]{Remark~}

\theoremstyle{definition}
\newtheorem{defn}[thm]{Definition~}

\newcommand{\KK}{\mathbb{K}}
\newcommand{\ZZ}{\mathbb{Z}}

\newcommand{\PP}{\mathbb{P}}

\newcommand{\CC}{\mathbb{C}}

\newcommand\PGL{\mathrm{PGL}}

\newcommand\Hom{\mathrm{Hom}}

\newcommand\I{\mathrm{I}}

\newcommand\SL{\mathrm{SL}}

\newcommand\diag{\mathrm{diag}}
\newcommand\GL{\mathrm{GL}}

\newcommand\ord{\mathrm{ord}}

\newcommand\Lin{\mathrm{Lin}}
\newcommand{\Map}{\mathrm{Map}}

\DeclareMathOperator{\Aut}{Aut}
\DeclareMathOperator{\res}{res}
\DeclareMathOperator{\corr}{cor}

\title{Sylow Criteria for Liftability of Automorphism Groups
	of Smooth Hypersurfaces}

\author[B. Xie]{Baiting Xie}
\address{Tsinghua University, China}
\email{xbt23@mails.tsinghua.edu.cn}

\author[Z. Zheng]{Zhiwei Zheng}
\address{Tsinghua University, China}
\email{zhengzhiwei@mail.tsinghua.edu.cn}
\date{}

\begin{document}
	\bibliographystyle{amsalpha}
	
		\begin{abstract} 
			In this paper, we first establish Sylow criteria for liftability of finite subgroups of the projective linear groups over arbitrary field. We also obtain Sylow criteria for $F$-liftability: if $F$ is a nonsingular polynomial of degree $d$ in $N$ variables over a field of characteristic zero, then a finite subgroup $G$ of $\Lin(F)$ is $F$-liftable if and only if for every prime $p$ dividing $\gcd(|G|,N,d)$, there exists a Sylow $p$-subgroup of $G$ that is $F$-liftable. Our proof combines restriction and corestriction in group cohomology with the reduction-to-Klein method.
		\end{abstract}
	
	\maketitle
	

	\section{Introduction}
\label{section: introduction}
Let $\KK$ be an arbitrary field, $V$ be an $N$-dimensional $\KK$-vector space, and
$F\in S^d(V^*)$ be a nonzero homogeneous form of degree $d\ge 1$.
We write $\Lin(F)$ for the group of projective linear transformations
preserving $F$ up to a nonzero scalar. The zero locus
$X=V(F)\subseteq\PP(V)$ is a 
hypersurface preserved by $\Lin(F)$. If $\KK$ is an algebraically closed field of characteristic zero, $N \geq 4$, $d \geq 3$, $(N,d) \neq (4,4)$, and $F$ is nonsingular, the theorem of Matsumura and
Monsky \cite{matsumura1963automorphisms} identifies $\Aut(X)$ with
$\Lin(F)$ and shows that this group is finite. Determining such groups remains difficult even for fixed $(d,N)$; see, for
example, \cite{oguiso2019quintic,wei2020automorphism,laza2022automorphisms,yang2024automorphism} when $\KK=\CC$.

We study a basic lifting problem for finite projective linear groups. A
finite subgroup $G\subseteq\PGL(V)$ is \emph{liftable} if it is the
isomorphic image of a finite subgroup of $\GL(V)$. When
$G\subseteq\Lin(F)$, we ask the stronger question whether such a lifting can be
chosen to preserve $F$ itself; in this case $G$ is called
\emph{$F$-liftable}. These notions are recalled precisely in
\S\ref{section: pre}. Over $\CC$, Gonz\'{a}lez-Aguilera, Liendo,
and Montero proved the corresponding $F$-liftability result when
$\gcd(N,d)=1$ \cite{gonzalez2020lift}.

Next we state our main results:
\begin{thm}
	\label{main theorem: liftability and Sylow subgroups}
	Let $\KK$ be an arbitrary field, let $V$ be an $N$-dimensional
	$\KK$-vector space, and let $G$ be a finite subgroup of $\PGL(V)$.
	Then the following statements hold.
	\begin{enumerate}
		\item The group $G$ is liftable if and only if, for every prime
		$p$ dividing $\gcd(|G|,N)$, it has a liftable Sylow
		$p$-subgroup.
		\item Let $F\in S^d(V^*)$ be a nonzero homogeneous form of degree
		$d\geq1$. If $G\subseteq\Lin(F)$, then $G$ is liftable if and only
		if, for every prime $p$ dividing $\gcd(|G|,N,d)$, it has a
		liftable Sylow $p$-subgroup.
	\end{enumerate}
\end{thm}

\begin{thm}
	\label{theorem: F-liftability and Sylow subgroups}
	Let $\KK$ be an arbitrary field, let $V$ be an $N$-dimensional $\KK$-vector space, and let $F\in S^{d}(V^{*})$ be a nonzero homogeneous form of degree $d\geq 1$. If $G$ is a finite subgroup of $\Lin(F)$, then $G$ is $F$-liftable if and only if, for every prime $p$ dividing $|G|$, the group $G$ has an $F$-liftable Sylow $p$-subgroup.
\end{thm}

\begin{thm}
	\label{main theorem: F-liftability and Sylow subgroups}
		Let $\KK$ be an arbitrary field of characteristic zero, let $V$ be an $N$-dimensional $\KK$-vector space, and let $F\in S^{d}(V^{*})$ be a nonsingular homogeneous form of degree $d\geq 1$. If $G$ is a finite subgroup of $\Lin(F)$, then $G$ is $F$-liftable if and only if, for every prime $p$ dividing $\gcd(|G|,N,d)$, the group $G$ has an $F$-liftable Sylow $p$-subgroup.
\end{thm}
For $F$-liftability, the condition in Theorem \ref{main theorem: F-liftability and Sylow subgroups} is empty whenever $\gcd(|G|,N,d)=1$.
Specializing to the case $\KK=\CC$ and $\gcd(N,d)=1$, the preceding theorem provides an alternative proof of the liftability
assertion in \cite[Theorem 3.5]{gonzalez2020lift}.

The proof has two main ingredients. First, the obstruction to lifting a
finite subgroup of $\PGL(V)$ is a class in $H^2(G,\KK^\times)$.
Restriction and corestriction reduce its vanishing to Sylow subgroups.
Once a lifting is fixed, the additional obstruction to preserving $F$
is encoded by a character of $G$, and the same formalism gives
the corresponding Sylow reduction for $F$-liftability. Second, for the
remaining Sylow subgroups we construct finite preimages whose
kernels have order coprime to $p$; taking Sylow $p$-subgroups of these
preimages then gives the required liftings and $F$-liftings. 
The main single-element input is obtained from the reduction-to-Klein
method: using \cite[Theorem 4.2]{zheng2022abelian}, we replace a
semi-invariant nonsingular form by a K-pure polynomial with the same
diagonal action and then use the explicit diagonal automorphism groups
of its Klein components.

\medskip
\noindent\emph{Acknowledgements.}
We thank Álvaro Liendo for his interest in this work and for his helpful
comments on an early draft.


\section{Preliminary Results}
\label{section: pre}
In this section we introduce some preliminary terminology and results that will be used later. The results are not new, but are organized in a way that suits our exposition.

\subsection{Automorphism Groups of Homogeneous Polynomials}
We first introduce the definitions of liftability and $F$-liftability following \cite[\S 4]{oguiso2019quintic}.
For an $N$-dimensional $\KK$-vector space $V$, we denote its general linear group and projective linear group by $\GL(V)$ and $\PGL(V)$, respectively. We denote by $\pi\colon\GL(V) \rightarrow \PGL(V)$ the canonical projection map. 

\begin{defn}
	\label{definition: liftability}
	Let $G$ be a subgroup of $\PGL(V)$. If there exists a subgroup $\widetilde{G} \subseteq \GL(V)$ such that the restriction of $\pi$ to $\widetilde{G}$ induces an isomorphism $\pi|_{\widetilde{G}}\colon\widetilde{G} \rightarrow G$, we say that $G$ is liftable and $\widetilde{G}$ is a lifting of $G$. In particular, for any $g \in \PGL(V)$, we say that $g$ is liftable if the subgroup $\langle g\rangle$ generated by $g$ is liftable. In this case, $\langle g\rangle$ has a lifting $\widetilde{G}$ and we say $\widetilde{g}=(\pi|_{\widetilde{G}})^{-1}(g)$ is a lifting of $g$. It follows that the liftings of $g$ are exactly the elements in $\pi^{-1}(g)$ with the same order as $g$. 
\end{defn}

For any nonzero homogeneous form $F \in S^{d}(V^{*})$ of degree $d \geq 1$, the linear automorphism group $\Lin(F)$ of $ F $ is defined as
\begin{equation*}
	\Lin(F)=\{\pi(\varphi)\mid \varphi\in\GL(V),\ \exists\lambda\in\KK^\times,\ F\circ\varphi = \lambda F\}.
\end{equation*}

Now we are ready to define $F$-liftability.
\begin{defn}
	\label{definition: F-liftability}
	Let $F \in S^{d}(V^{*})$ be a nonzero homogeneous form of degree $d \geq 1$. 
	\begin{enumerate}
		\item We say that a single element $ \varphi \in \GL(V) $ is $ F $-preserving if $F\circ \varphi = F$. We say that a subgroup $ H $ of $ \GL(V) $ is $ F $-preserving if all elements in $ H $ are $ F $-preserving.
		\item 	Let $G$ be a subgroup of $\Lin(F)$. If there exists a lifting $\widetilde{G}$ of $G$ such that $ \widetilde{G} $ is $ F $-preserving, we say that $G$ is $F$-liftable and $\widetilde{G}$ is an $F$-lifting of $G$. In particular, for any $g \in \Lin(F)$, we say that $g$ is $F$-liftable if $g$ has a lifting $\widetilde{g}$ such that $F \circ \widetilde{g} = F$. In this case, we say $\widetilde{g}$ is an $F$-lifting of $g$.
	\end{enumerate}
\end{defn}

Let $X \subset \PP(V)$ be a smooth hypersurface of degree $d \geq 1$. Then $X$ is the zero locus of some homogeneous form $F \in S^{d}(V^{*})$.

\begin{defn}
	\label{definition: nonsingular}
	A nonzero homogeneous form $F \in S^{d}(V^{*})$ is called \textbf{nonsingular} if for some (hence every) basis $\{x_{1},x_{2},\cdots,x_{N}\}$ of $V^{*}$, the partial derivatives $\frac{\partial F}{\partial x_{1}},\frac{\partial F}{\partial x_{2}},\cdots,\frac{\partial F}{\partial x_{N}}$ and $ F $ have no common zero in $ \PP^{N-1}_{\overline{\KK}} $, where $ \overline{\KK} $ is the algebraic closure of $ \KK $.
\end{defn}

When $\KK$ is algebraically closed, the hypersurface $X$ is smooth if
and only if its defining form $F$ is nonsingular. Let $\Aut(X)$ denote
the group of regular automorphisms of $X$. If $N\geq3$ and $d\geq2$,
restriction to $X$ induces an injective homomorphism $\Lin(F)\longrightarrow\Aut(X)$. The following classical results of Matsumura and Monsky for
$N\geq4$, and of Chang for $N=3$, identify this map as an isomorphism,
apart from two exceptional pairs, and also give the finiteness of the
automorphism group.

\begin{thm}[{\cite[Theorems 1 and 2]{matsumura1963automorphisms};
	\cite{chang1978plane}}]
	\label{theorem: Aut=Lin}
	Let $\KK$ be an algebraically closed field of characteristic zero. Let $X$ be a smooth hypersurface defined by a homogeneous form $F$ of degree $d \geq 3$ in the projective space $\PP(V)$ over $\KK$ of dimension $N-1$. If $(N,d) \neq (3,3),(4,4)$, $N \geq 3$, $d \geq 3$, then $\Aut(X)=\Lin(F)$ is a finite subgroup of $\PGL(V)$.
\end{thm}

The notion of $F$-liftability applies to subgroups of $\Lin(F)$ for any nonzero homogeneous form $F$. In the setting of Theorem \ref{theorem: Aut=Lin}, $\Aut(X)$ and $\Lin(F)$ are identified, hence $F$-liftability is defined for subgroups of $\Aut(X)$ as well.  From now on, we will concentrate on lifting questions for subgroups of $\Lin(F)$, which are more general.

\subsection{Group Cohomology}
We use two standard functorial operations in group cohomology:
restriction to a subgroup and corestriction back to the ambient group.
The latter transfers a cohomology class from a finite-index subgroup by
summing its translates over the cosets. We recall a precise construction
and the relation between these two operations.

Let $G$ be a group and let $M$ be a left $G$-module, written
additively. For $n\geq0$, define
\[
C^n(G,M)=\Map(G^n,M)
\]
and the differential $\mathrm d^n\colon C^n(G,M)\to C^{n+1}(G,M)$ by
\begin{align*}
	(\mathrm d^n f)(g_1,\ldots,g_{n+1})
	={}&g_1\cdot f(g_2,\ldots,g_{n+1}) \\
	&+\sum_{i=1}^n(-1)^i
	f(g_1,\ldots,g_i g_{i+1},\ldots,g_{n+1}) +(-1)^{n+1}f(g_1,\ldots,g_n).
\end{align*}
Then we obtain a cochain complex $C^\bullet(G,M)$ given by
\[
0\longrightarrow C^0(G,M)\xrightarrow{\mathrm d^0}C^1(G,M)
\xrightarrow{\mathrm d^1}C^2(G,M)\longrightarrow\cdots.
\]
The $ n $-th cohomology $ H^{n}(G,M) $ of $ G $ with coefficients in $ M $ is then defined to be the $ n $-th cohomology group of $C^\bullet(G,M)$. More explicitly, $H^0(G,M)=M^G$, while for $n\geq1$,
\[
H^n(G,M)=\ker(\mathrm d^n)/\operatorname{im}(\mathrm d^{n-1}).
\]

Let $K\subseteq G$ be a subgroup, and regard $M$ as a $K$-module by
restriction. Restricting cochains from $G^n$ to $K^n$ induces
\[
\res_K^G\colon H^n(G,M)\longrightarrow H^n(K,M).
\]
When $K$ has finite index, the \emph{corestriction} is a homomorphism in
the opposite direction. To define it, consider the coinduced
$G$-module
\[
\operatorname{Coind}_K^G M
=\operatorname{Hom}_{\ZZ K}(\ZZ G,M),
\]
where $\ZZ G$ is regarded as a left $\ZZ K$-module. This module may be
viewed as the set of maps $\varphi\colon G\to M$ satisfying
\[
\varphi(kx)=k\cdot\varphi(x)\qquad(k\in K,\ x\in G).
\]
Its $G$-action is given by
$(g\cdot\varphi)(x)=\varphi(xg)$. Choose a set $T$ of representatives
for the right cosets $K\backslash G$. The coset trace
\[
\operatorname{Tr}(\varphi)
=\sum_{t\in T}t^{-1}\cdot\varphi(t)
\]
is independent of the representatives: replacing $t$ by $kt$ does not
change the summand because $\varphi(kt)=k\cdot\varphi(t)$. Moreover,
it defines a $G$-module homomorphism
\[
\operatorname{Tr}\colon\operatorname{Coind}_K^G M\longrightarrow M.
\]
Shapiro's lemma gives an isomorphism
\[
\operatorname{Sh}\colon
H^n(G,\operatorname{Coind}_K^G M)\xrightarrow{\sim}H^n(K,M),
\]
and the corestriction is defined by
\[
\corr_K^G
:=\operatorname{Tr}_*\circ\operatorname{Sh}^{-1}
\colon H^n(K,M)\longrightarrow H^n(G,M).
\]
In degree zero, this is the usual
norm map $M^K\to M^G$, given by
$m\mapsto\sum_{t\in T}t^{-1}\cdot m$; indeed, $T^{-1}$ is a set of
representatives for the left cosets $G/K$.

The corestriction is also called the cohomological transfer. The first
published treatment of the cohomological transfer is due to Eckmann
\cite[Theorem~7]{Eckmann1953Transfer}; see also Brown for the historical
background and a modern account
\cite[Chapter~III, \S9]{Brown1982Cohomology}. In particular,
\cite[Chapter~III, Proposition~9.5(ii)]{Brown1982Cohomology} gives the
following fundamental relation between restriction and corestriction.
\begin{lem}
	\label{lemma: res and cor}
	Let $G$ be a finite group and $K\subseteq G$ a subgroup. For every
	$z\in H^n(G,M)$,
	\[
	\corr_K^G\res_K^G z=[G:K]\cdot z.
	\]
	In particular, taking $K=\{1\}$ shows that $|G|$ annihilates
	$H^n(G,M)$ for every $n\geq1$.
\end{lem}

We now specialize to the trivial $G$-module $\KK^\times$, which is the
coefficient group relevant to projective representations. In this case,
we write the cohomology groups multiplicatively. The low-degree cohomology
groups then have the following explicit descriptions:

$H^{1}(G,\KK^\times) = \Hom(G,\KK^\times)$, and $H^{2}(G,\KK^\times)$ is the quotient $Z^{2}(G,\KK^\times) / B^{2}(G,\KK^\times)$, where
\begin{align*}
	Z^{2}(G,\KK^\times)
	={}&\bigl\{c\colon G^2\to\KK^\times\mid
	c(h,k)c(g,hk)=c(g,h)c(gh,k)\bigr\},\\
	B^{2}(G,\KK^\times)
	={}&\bigl\{c\colon G^2\to\KK^\times\mid
	c(g,h)=f(gh)^{-1}f(g)f(h)\text{ for some }f\colon G\to\KK^\times\bigr\}.
\end{align*}
Under these explicit descriptions, the map $\res_{K}^{G}$ is simply the restriction of functions.

\section{Liftability and $ p $-Subgroups}
\label{section: sylow subgroup}

In this section, we use
group cohomology to reduce the lifting problems to Sylow $p$-subgroups. We first recall the relation between group cohomology and liftability. Fix for each $g \in G$ an element $L(g) \in \pi^{-1}(g)$. Then there exists a unique scalar $c(g,h) \in \KK^\times$ such that $L(gh)=c(g,h)L(g)L(h)$. Then $c$ satisfies:
\begin{equation*}
	\forall g,h,k \in G,\  c(h,k)c(g,hk)=c(g,h)c(gh,k).
\end{equation*}

So $c \in Z^{2}(G,\KK^\times)$. This $c$ depends on the choice of $L$. For another choice of preimage $L'(g) = f(g)L(g)$, where $f(g) \in \KK^\times$, suppose it defines $c' \in Z^{2}(G,\KK^\times)$. Notice that $f \in \Map(G,\KK^\times)$ satisfies: 
\begin{equation*}
	\forall g,h \in G,\ c'(g,h)=f(gh)f(g)^{-1}f(h)^{-1}c(g,h)
\end{equation*}

Thus $L$ defines a unique class $o_{G} = [c] $ in $H^{2}(G,\KK^\times)$, which is independent of the choice of $L$. $G$ is liftable if and only if $o_{G} = 1$. 

The following theorem reduces liftability of an arbitrary finite subgroup to that of its Sylow subgroups.

\begin{thm}
	\label{theorem: liftability and Sylow subgroups}
	Let $\KK$ be an arbitrary field, let $V$ be a finite-dimensional $\KK$-vector space, and let $G$ be a finite subgroup of $\PGL(V)$. Then $G$ is liftable if and only if for any prime $p$ dividing $|G|$, there exists a Sylow $p$-subgroup $G_{p}$ of $G$ such that $G_{p}$ is liftable. 
\end{thm}

\begin{proof}
		We only need to prove the sufficiency. 
		If $G=\{1\}$, the conclusion is immediate, so we may assume
		$|G|>1$.

	By definition of $o_{G}$ and $\res_{G_{p}}^{G}\colon H^{2}(G,\KK^\times) \rightarrow H^{2}(G_{p},\KK^\times)$, we have $o_{G_{p}} = \res_{G_{p}}^{G}o_{G}$. Since $G_{p}$ is liftable, we have $ \res_{G_{p}}^{G}o_{G} = o_{G_{p}}=1$. So by Lemma~\ref{lemma: res and cor}, for any prime $p$ dividing $|G|$, we have
	\begin{equation*}
		(o_{G})^{\frac{|G|}{|G_{p}|}} = \corr_{G_{p}}^{G}\res_{G_{p}}^{G}o_{G} = \corr_{G_{p}}^{G}(1) = 1,
	\end{equation*}
		For $k_p=|G|/|G_p|$, we have
		$\gcd_{p\mid |G|}k_p=1$. Hence there are integers $l_p$ such that
		$\sum_{p\mid |G|}k_pl_p=1$, and therefore
		\[
		o_G=\prod_{p\mid |G|}\bigl(o_G^{k_p}\bigr)^{l_p}=1.
		\]
		Thus $G$ is liftable.
\end{proof}

\begin{lem}
	\label{lemma: liftability when gcd order N=1}
	Let $\KK$ be an arbitrary field, let $V$ be an $N$-dimensional
	$\KK$-vector space, and let $G$ be a finite subgroup of $\PGL(V)$.
	If $\gcd(|G|,N)=1$, then $G$ is liftable.
\end{lem}

\begin{proof}
	For each $g\in G$, choose $L(g)\in\pi^{-1}(g)$, and define
	$c(g,h)\in\KK^\times$ by
	\[
	L(gh)=c(g,h)L(g)L(h).
	\]
	Recall that $o_G=[c]\in H^2(G,\KK^\times)$.
	Since
	\begin{equation*}
		c(g,h)^{N} = \frac{\det(L(gh))}{\det(L(g))\det(L(h))},
	\end{equation*}
	we have $ o_{G}^{N} = 1 $. By Lemma \ref{lemma: res and cor} we have $ o_{G}^{|G|} = 1 $. Since $ \gcd(|G|,N) = 1 $, we have $ o_{G}=1 $. So $ G $ is liftable.
\end{proof}

\begin{lem}
	\label{lemma: liftability when gcd order N d=1}
	Let $\KK$ be an arbitrary field, let $V$ be an $N$-dimensional
	$\KK$-vector space, and let $F\in S^d(V^*)$ be a nonzero homogeneous
	form of degree $d\geq1$. If $G$ is a finite subgroup of $\Lin(F)$ and
	$\gcd(|G|,N,d)=1$, then $G$ is liftable.
\end{lem}

\begin{proof}
	As in the proof of Lemma~\ref{lemma: liftability when gcd order N=1},
	choose $L(g)\in\pi^{-1}(g)$ and let $c$ be the resulting cocycle, so
	that $o_G=[c]\in H^2(G,\KK^\times)$.
	Since $G\subseteq\Lin(F)$, for each $g\in G$ there exists
	$\gamma(g)\in\KK^\times$ such that $F\circ L(g)=\gamma(g)F$. Then
	\begin{equation*}
		c(g,h)^{d} = \frac{\gamma(gh)}{\gamma(g)\gamma(h)}.
	\end{equation*}
So $ o_{G}^{d} = 1 $. As in the proof of Lemma \ref{lemma: liftability when gcd order N=1}, we have $ o_{G}^{N} = o_{G}^{|G|} = 1 $. Since $ \gcd(|G|,N,d)=1 $, we have $ o_{G}=1 $. So $ G $ is liftable.
\end{proof}

We now combine Theorem~\ref{theorem: liftability and Sylow subgroups}
with Lemmas~\ref{lemma: liftability when gcd order N=1} and
\ref{lemma: liftability when gcd order N d=1} to prove
Theorem~\ref{main theorem: liftability and Sylow subgroups}.

\begin{proof}[Proof of Theorem~\ref{main theorem: liftability and Sylow subgroups}]
	The necessity in both statements is immediate. So we only need to prove the sufficiency.
	
	For (1), assume that
		$G$ has a liftable Sylow $p$-subgroup for every prime $p$ dividing
		$\gcd(|G|,N)$. If $p\mid |G|$ but $p\nmid N$, then a Sylow
	$p$-subgroup $G_p$ satisfies $\gcd(|G_p|,N)=1$ and is therefore
	liftable by Lemma~\ref{lemma: liftability when gcd order N=1}.
	Thus $G$ has a liftable Sylow $p$-subgroup for every prime $p\mid |G|$,
	and Theorem~\ref{theorem: liftability and Sylow subgroups} shows that
	$G$ is liftable.

	For (2), suppose that $G\subseteq\Lin(F)$ and that $G$ has a liftable
		Sylow $p$-subgroup for every prime $p$ dividing $\gcd(|G|,N,d)$. If
		$p\mid |G|$ but $p\nmid\gcd(N,d)$, then a Sylow $p$-subgroup
	$G_p$ satisfies $\gcd(|G_p|,N,d)=1$ and is therefore liftable by
	Lemma~\ref{lemma: liftability when gcd order N d=1}.
	Another application of
	Theorem~\ref{theorem: liftability and Sylow subgroups} completes the
	proof.
\end{proof}

We now turn to the $F$-liftability analogue of Theorem~\ref{theorem: liftability and Sylow subgroups}. The following Lemma reduces divisibility of an element in $ \Hom(G,\KK^\times)$ to that of the Sylow subgroups of $ G $.

\begin{lem}
	\label{lemma: d-divisibility and Sylow subgroup}
	Let $\KK$ be an arbitrary field, let $d \geq 1$ be an integer, and let $G$ be a finite group. For any prime $p$ dividing $|G|$, let $G_{p}$ be a Sylow $p$-subgroup of $G$. Assume $\gamma \in \Hom(G,\KK^\times)$ satisfies that for any $p$ dividing $|G|$, there exists $\delta_{p} \in \Hom(G_{p},\KK^\times)$ such that $ \gamma|_{G_{p}} = \delta_{p}^{d}$, then there exists $\delta \in \Hom(G,\KK^\times)$ such that $\gamma = \delta^{d}$.  
\end{lem}
\begin{proof}
		If $G=\{1\}$, the conclusion is immediate. We may therefore assume
		that $|G|>1$. For each prime $p$ dividing $|G|$, denote by
		$k_{p}=|G|/|G_{p}|$. By Lemma \ref{lemma: res and cor}, there exists a
	group homomorphism
	\[
	\corr_{G_{p}}^{G}\colon
	\Hom(G_{p},\KK^\times)\longrightarrow\Hom(G,\KK^\times)
	\]
	such that
	\[
	\corr_{G_{p}}^{G}(\gamma|_{G_{p}})
	=\corr_{G_{p}}^{G}\res_{G_{p}}^{G}\gamma
	=\gamma^{k_{p}}.
	\]
	
	Since the greatest common divisor of $k_{p}$, for $p\mid |G|$, is
	$1$, there exist integers $l_{p}$, for $p\mid |G|$, such that
	\[
	\sum_{p\mid |G|}k_{p}l_{p}=1.
	\]
	Then taking
	\[
	\delta
	=\prod_{p\mid |G|}
	\bigl(\corr_{G_{p}}^{G}\delta_{p}\bigr)^{l_{p}}
	\in\Hom(G,\KK^\times),
	\]
	we have
	\begin{equation*}
		\delta^{d}
		=\prod_{p\mid |G|}
		\bigl(\corr_{G_{p}}^{G}\delta_{p}^{d}\bigr)^{l_{p}}
		=\prod_{p\mid |G|}
		\bigl(\corr_{G_{p}}^{G}(\gamma|_{G_{p}})\bigr)^{l_{p}}
		=\prod_{p\mid |G|}\gamma^{k_{p}l_{p}}
		=\gamma.
	\end{equation*}
\end{proof}

We now use the preceding lemma to prove the general $F$-liftability
criterion stated in the introduction.

\begin{proof}[Proof of Theorem~\ref{theorem: F-liftability and Sylow subgroups}]
	We only need to prove the sufficiency. 
	
	By Theorem \ref{theorem: liftability and Sylow subgroups}, there exists a lifting $\widetilde{G}$ of $G$. Denote by $\iota$ the inverse of $\pi|_{\widetilde{G}}\colon\widetilde{G} \rightarrow G$.  For each $g \in G$, there exists $\gamma(g) \in \KK^\times $ such that $ F \circ \iota(g) = \gamma(g) F $. Then we obtain a group homomorphism $\gamma \colon G \rightarrow \KK^\times$.
	
	For any prime $p$ dividing $|G|$, there exists an $F$-liftable Sylow $p$-subgroup $G_{p}$ of $G$. Choose an $F$-lifting $\widetilde{G_{p}}$ of $G_{p}$ and denote by $\iota_{p}$ the inverse of $\pi|_{\widetilde{G_{p}}}$. Note that both $\iota(G_{p}) $ and $ \widetilde{G_{p}} $ are liftings of $ G_{p} $. So there exists $\delta_{p} \in \Hom(G_{p},\KK^\times)$ such that $\iota(g) = \delta_{p}(g)\iota_{p}(g)$ for any $g \in G_{p}$. So $F \circ \iota(g) = \delta_{p}^{d}(g)(F \circ \iota_{p}(g)) =  \delta_{p}^{d}(g)F$ for any $g \in G_{p}$, which implies $\gamma|_{G_{p}} = \delta_{p}^{d} $. Then by Lemma \ref{lemma: d-divisibility and Sylow subgroup}, there exists $\delta \in \Hom(G,\KK^\times)$ such that $\gamma = \delta^{d}$. 
	
	We then define the following group homomorphism:
	\begin{equation*}
		\widetilde{\iota}\colon G \rightarrow \GL(V),\ \widetilde{\iota}(g)=\delta(g)^{-1}\iota(g).
	\end{equation*}
	
	Since $ (\pi \circ \widetilde{\iota})(g) = g $ and $F \circ \widetilde{\iota}(g) = \delta(g)^{-d}(F \circ \iota(g)) =  F $ for any $g \in G$, we deduce that $\widetilde{\iota}(G)$ is an $F$-lifting of $G$. So $G$ is $F$-liftable.
\end{proof}

\section{Liftability of a Single Element}
\label{section: single element}
In this section we consider the lifting question for a single element
of finite order. We begin with an elementary observation that will be
used repeatedly.

\begin{lem}
	\label{lemma: lifting finite-order projective element}
	Let $\KK$ be an algebraically closed field, let $V$ be a
	finite-dimensional $\KK$-vector space, and let $g\in\PGL(V)$ have
	finite order $q$. Then $g$ admits a lifting of order $q$.
\end{lem}

\begin{proof}
	Choose $A_0\in\pi^{-1}(g)$. Since $g^q=1$, we have
	$A_0^q=c\I_V$ for some $c\in\KK^\times$. Choose
	$\lambda\in\KK^\times$ such that $\lambda^q=c^{-1}$. Then
	$(\lambda A_0)^q=\I_V$. Its order is exactly $q$ because its
	projective image has order $q$.
\end{proof}

We first recall the terminology of \cite[\S2]{zheng2022abelian}.
\begin{defn}
	\label{definition: K-pure polynomial}
	For $k\geq1$ and $d\ge 1$, a polynomial of the form
	\[
	x_1^{d-1}x_2+x_2^{d-1}x_3+\cdots+x_k^{d-1}x_1
	\]
	is called a \emph{Klein polynomial} in $k$ variables; when $k=1$,
	it is the Fermat monomial $x_1^d$. A homogeneous polynomial is called
	\emph{K-pure} if it is a sum of Klein polynomials in mutually disjoint
	sets of variables, with every variable occurring in exactly one summand.
\end{defn}

The next lemma allows us to replace a nonsingular polynomial by a K-pure
polynomial while preserving an abelian diagonal action and its
semi-invariant character. Thus questions depending only on this action can
be reduced to the corresponding questions for Klein polynomials. We call
this procedure the \emph{reduction-to-Klein method}. The statement and
proof are adapted from \cite[Lemma~4.1 and Theorem~4.2]{zheng2022abelian};
although those results are stated over $\CC$, their arguments are
algebraic. We include the argument in the form needed here.

\begin{lem}
	\label{lemma: reduction to klein}
	Let $\KK$ be an algebraically closed field of characteristic zero, let $V$ be an $N$-dimensional $\KK$-vector space, and let $H\subseteq\GL(V)$ be a finite abelian group such that
	$\gcd(|H|,d-1)=1$. Suppose that $F\in S^d(V^*)$ is nonsingular and $H$-semi-invariant, namely,
	there is a character $\chi\colon H\to\KK^\times$ such that
	\[
	F\circ h=\chi(h)F\qquad(h\in H).
	\]
	Then one can choose coordinates $x_1,\ldots,x_N$ for $V$ in which $H$
	is diagonal, and there exists a K-pure polynomial $\widehat F$ of degree $ d $ such that
	\[
	\widehat F\circ h=\chi(h)\widehat F\qquad(\forall h\in H).
	\]
\end{lem}

\begin{proof}
		When $d=1$, the condition $\gcd(|H|,d-1)=1$ forces $H=\{1\}$. 
		After choosing any coordinates $x_1,\ldots,x_N$, the polynomial
		$\widehat F=x_1+\cdots+x_N$ is K-pure and $H$-invariant, so the
		conclusion is immediate. We may therefore assume $d\ge 2$. Since $H$ is finite and abelian, its elements can be diagonalized
	simultaneously. Choose coordinates such that
	\[
	h=\diag\bigl(\lambda_1(h),\ldots,\lambda_N(h)\bigr)
	\qquad(h\in H),
	\]
	where $\lambda_1,\ldots,\lambda_N\in\Hom(H,\KK^\times)$. Every monomial
	of $F$ with nonzero coefficient then has $H$-character $\chi$.
	
	We first recall the consequence of nonsingularity used in \cite[Lemma 3.2]{oguiso2019quintic} and 
	\cite[Lemma~4.1]{zheng2022abelian}. If $I,J\subseteq\{1,\ldots,N\}$
	are disjoint and $|I|>|J|$, then $F$ contains a monomial whose total
	degree in the variables indexed by $I$ is at least $d-1$ and whose
	total degree in the variables indexed by $J$ is zero. Indeed, otherwise
	the sum of the terms of $F$ having $I$-degree at least $d-1$ would have
	the form
	\[
	\sum_{j\in J}Q_j(x_i\mid i\in I)x_j,
	\]
	where the $Q_j$ are homogeneous of degree $d-1$. Since $|J|<|I|$,
	the equations $Q_j=0$ have a common nonzero solution in the variables
	indexed by $I$. Setting all other variables equal to zero then gives a
	singular point of $F$, a contradiction.
	
	For a character $\alpha\in\widehat H=\Hom(H,\KK^\times)$, set
	\[
	I_\alpha=\{i\mid\lambda_i=\alpha\},
	\]
	and define
	\[
	\tau(\alpha)=\chi\alpha^{1-d}.
	\]
	Since $\gcd(|H|,d-1)=1$, the power map
	$\alpha\mapsto\alpha^{1-d}$ is an automorphism of $\widehat H$;
	hence $\tau$ is a permutation of $\widehat H$.
	
	We claim that
	\[
	|I_{\tau(\alpha)}|\ge |I_\alpha|
	\]
	for every $\alpha$. This is clear if $\tau(\alpha)=\alpha$. Otherwise,
	suppose that $|I_{\tau(\alpha)}|<|I_\alpha|$ and apply the preceding
	consequence of nonsingularity with
	$I=I_\alpha$ and $J=I_{\tau(\alpha)}$. We obtain a monomial of $F$
	with $I$-degree at least $d-1$ and $J$-degree zero. Its $I$-degree
	cannot be $d$, since
	\[
	\chi=\alpha^{d-1}\tau(\alpha)
	\]
	and $\tau(\alpha)\ne\alpha$. It therefore has $I$-degree $d-1$,
	and its remaining variable must have character $\tau(\alpha)$ because
	the monomial has character $\chi$. This contradicts its $J$-degree
	being zero, and proves the claim.
	
	Along each orbit of the permutation $\tau$, these inequalities force
	all the cardinalities $|I_\alpha|$ to be equal. We may therefore choose
	bijections
	\[
	I_\alpha\longrightarrow I_{\tau(\alpha)}
	\]
	for all $\alpha$. Together they define a permutation $\sigma$ of
	$\{1,\ldots,N\}$ satisfying
	\[
	\lambda_i^{d-1}\lambda_{\sigma(i)}=\chi
	\qquad(1\le i\le N).
	\]
	Consequently,
	\[
	\widehat F=\sum_{i=1}^N x_i^{d-1}x_{\sigma(i)}
	\]
	is $H$-semi-invariant with character $\chi$. Finally, decomposing
	$\sigma$ into disjoint cycles expresses $\widehat F$ as a sum of Klein
	polynomials in mutually disjoint sets of variables. Thus $\widehat F$
	is K-pure.
\end{proof}

The following description of diagonal automorphisms of a Klein polynomial is
given in \cite[Lemma 3.1]{zheng2022abelian}. We include the proof for the
convenience of the reader.

\begin{lem}
	\label{lemma: aut of klein}
	Let $\KK$ be an arbitrary field of characteristic zero. Let
	\[
	K=x_1^{d-1}x_2+x_2^{d-1}x_3+\cdots+x_n^{d-1}x_1\in\KK[x_1,\ldots,x_n]
	\]
	be a Klein polynomial. If $A=\diag(a_1,\ldots,a_n)$ satisfies
	$K\circ A=K$, then
	\[
	a_i=a_1^{(1-d)^{i-1}}\quad(1\le i\le n),
	\qquad a_1^{1-(1-d)^n}=1,
	\]
	and $\det(A)^d=1$.
\end{lem}

\begin{proof}
	Comparing the coefficients of the monomials in $K\circ A=K$ gives
	$a_i^{d-1}a_{i+1}=1$, where the indices are read modulo $n$.
	Thus $a_{i+1}=a_i^{1-d}$, and iteration gives
	$a_i=a_1^{(1-d)^{i-1}}$. The last relation, with $i=n$, gives
	$a_1^{1-(1-d)^n}=1$. Finally,
	\[
	\det(A)^d
	=a_1^{d\sum_{i=0}^{n-1}(1-d)^i}
	=a_1^{1-(1-d)^n}=1.
	\]
\end{proof}

\begin{rmk}
	\label{remark: diagonal automorphisms of Klein polynomial}
	Assume that $\KK$ is an algebraically closed field of characteristic zero. For $d\ge 3$, the full group-theoretic description is given in
	\cite[Lemma 3.1]{zheng2022abelian}. If $M=|1-(1-d)^n|$, then the
	group of diagonal matrices in $\GL(n,\KK)$ preserving $K=x_1^{d-1}x_2+x_2^{d-1}x_3+\cdots+x_n^{d-1}x_1$ is cyclic
	of order $M$, generated by
	\[
	\diag\bigl(\zeta,\zeta^{1-d},\ldots,
	\zeta^{(1-d)^{n-1}}\bigr),
	\]
	where $\zeta$ is a primitive $M$-th root of unity. Its image in
	$\PGL(n,\KK)$ is the group of diagonal projective automorphisms of
	$V(K)$ and is cyclic of order $M/d$.
	
	When $d=2$ and $n$ is odd, the corresponding projective statement
	remains valid: every diagonal projective automorphism of
	$K=x_1x_2+\cdots+x_nx_1$ is trivial. In contrast, when $d=2$ and $n$
	is even, the group of diagonal projective automorphisms is infinite.
\end{rmk}

\begin{prop}
	\label{proposition: determinant equality}
	Let $\KK$ be an arbitrary field of characteristic zero, let $V$ be a finite-dimensional $\KK$-vector space, and let $F \in S^{d}(V^{*})$ be nonsingular with $d\ge 1$. Let $\varphi \in \GL(V)$ be an element of finite order such that $F \circ \varphi = F$
	and $\gcd(\ord(\varphi),d-1)=1$. Then $\det(\varphi)^d=1$.
\end{prop}
\begin{proof}
	By replacing $ \KK $ with its algebraic closure, we may assume $ \KK $ is algebraically closed. Applying Lemma \ref{lemma: reduction to klein} to
	$H=\langle\varphi\rangle$, we can choose coordinates $x_1,\ldots,x_N$ for $V$ in which $ \varphi $	is diagonal and preserves a K-pure polynomial of degree $ d $
	\[
	\widehat F=K_1+\cdots+K_t,
	\]  
	where the $K_j$ involve
	mutually disjoint sets of variables. The diagonal matrix of $\varphi$ restricts to a diagonal
	automorphism $\varphi_j$ preserving each $K_j$. Then Lemma
	\ref{lemma: aut of klein} yields $\det(\varphi_j)^d=1$ for every
	$j$. Hence $\det(\varphi)^d=\prod_{j=1}^t\det(\varphi_j)^d=1.$
\end{proof}

\begin{prop}
	\label{proposition: F-liftability when order coprime to d}
	Let $\KK$ be an arbitrary field, let $V$ be a finite-dimensional
	$\KK$-vector space, and let $F\in S^d(V^*)$ be a nonzero homogeneous
	form of degree $d\geq1$. Let $g\in\Lin(F)$ have finite order $q$. If
	$\gcd(q,d)=1$, then $g$ admits a unique $F$-lifting in $\GL(V)$.
\end{prop}

\begin{proof}
	Suppose first that $\KK$ is algebraically closed. By Lemma
	\ref{lemma: lifting finite-order projective element}, choose a lifting
	$A$ of $g$ of order $q$, and write $F\circ A=\zeta F$. Then
	$\zeta^q=1$. Choose integers $a,b$ such that $ad+bq=1$, and set
	$\xi=\zeta^{-a}$. We have $\xi^q=1$ and $\xi^d=\zeta^{-1}$.
	Consequently, $\widetilde g=\xi A$ satisfies
	\[
	\pi(\widetilde g)=g,\qquad
	\widetilde g^q=\I_V,\qquad
	F\circ\widetilde g=F.
	\]
	Its projective image has order $q$, so $\widetilde g$ also has order
	$q$ and is therefore an $F$-lifting of $g$.
	
	Now let $\KK$ be arbitrary and let $\overline{\KK}$ be its algebraic
	closure. The preceding argument gives an $F$-lifting $\widetilde g$
	over $\overline{\KK}$. Since
	$\gcd(q,\dim(V),d)=1$, Lemma
	\ref{lemma: liftability when gcd order N d=1} gives a lifting
	$C\in\GL(V)$ of $g$. Write $C=\mu\widetilde g$ with
	$\mu\in\overline{\KK}^\times$. Both matrices have order $q$, so
	$\mu^q=1$. Moreover,
	\[
	F\circ C=\mu^dF.
	\]
	Since $C$ and $F$ are defined over $\KK$ and $F\ne0$, we have
	$\mu^d\in\KK^\times$. As $\gcd(q,d)=1$, it follows that
	$\mu\in\KK^\times$, and hence $\widetilde g=\mu^{-1}C$ belongs to
	$\GL(V)$.
	
	Finally, if $\overline g$ is another $F$-lifting, write
	$\widetilde g=\lambda\overline g$. Then
	$\lambda^q=\lambda^d=1$, so $\lambda=1$.
\end{proof}

We next consider elements of prime-power order when the prime divides
the degree. Over $\CC$, the following result appears as
\cite[Lemma 3.1]{gonzalez2020lift}.\footnote{In the second sentence of the proof of \cite[Lemma 3.1]{gonzalez2020lift}, the authors invoke Proposition 1.15 to obtain an $F$-lifting. As stated, Proposition 1.15 does not seem to apply directly in the required generality, so an additional argument appears to be needed at this point. The proof below supplies such an argument.} 
\begin{prop}
	\label{proposition: F-liftability when p mid d and p nmid N}
	Let $\KK$ be an arbitrary field of characteristic zero, let $V$ be an $N$-dimensional $\KK$-vector space, and let $F \in S^{d}(V^{*})$ be a nonsingular homogeneous form of degree $d\ge 1$. Let $g \in \Lin(F)$ be an element of order $p^{r}$ with prime $p$ and $r \in \ZZ^+$. If $p \mid d $ but $ p \nmid N $, then there exists a unique $F$-lifting of $g$ in $\SL(V)$.
\end{prop}

\begin{proof}
	We first prove the existence when $ \KK $ is algebraically closed. By
	Lemma \ref{lemma: lifting finite-order projective element}, choose
	$A\in\pi^{-1}(g)$ of order $p^r$, and write
	$F\circ A=\zeta F$. Since $p\mid d$, we have
	$\gcd(p^r,d-1)=1$. Applying Lemma \ref{lemma: reduction to klein} to $ H = \langle A \rangle $, we can choose coordinates $x_1,\ldots,x_N$ for $V$ in which $ A $ is diagonal and there is a K-pure
	polynomial
	\[
	\widehat F=K_1+\cdots+K_t
	\]
	such that $\widehat F\circ A=\zeta \widehat F$, where the $K_j$ involve
	mutually disjoint sets of variables. Since their numbers of variables
	sum to $N$ and $p\nmid N$, some $K_j$, say $K_1$, involves $k$
	variables with $p\nmid k$.

	Let $A_1$ be the restriction of $A$ to these variables and set
	$M=|1-(1-d)^k|$. Since
	\[
	\frac{1-(1-d)^k}{d}
	=\sum_{j=1}^k\binom{k}{j}(-d)^{j-1}
	\equiv k\pmod p,
	\]
	the integer $M/d$ is coprime to $p$.
	Since $\KK$ is algebraically closed, choose $\eta\in\KK^\times$ such
	that $\eta^d=\zeta^{-1}$. Then $\eta A_1$ preserves $K_1$ and has the
	same projective class as $A_1$. If $d=2$, then $p=2$, and $p\nmid k$
	implies that $k$ is odd, so the second part of Remark
	\ref{remark: diagonal automorphisms of Klein polynomial} applies; if
	$d\geq3$, its first part applies. Thus, by the same remark, the
	projective class of $A_1$ belongs to a cyclic group of order $M/d$.
	It also has $p$-power order, and is therefore trivial. Hence
	$A_1=\xi\I$ for some $ p^{r} $-th root of unity $\xi\in\KK^\times$, and $K_1\circ A_1=\zeta K_1$ gives
		$\xi^d=\zeta$. Thus $B=\xi^{-1}A$ satisfies $F\circ B=F$ and $B^{p^r}=1$. Since $\pi(B)=g$ has order $p^r$, the order of $B$ is exactly $p^r$. Therefore $B$ is an $F$-lifting of $g$.

	By Proposition \ref{proposition: determinant equality},
	$\det(B)^d=1$; moreover,
	$\det(B)^{p^r}=1$. Put $m=\gcd(d,p^r)$. Since
	$p\nmid N$, there exists an $ m $-th root of unity $\omega$ such that
		$\omega^N=\det(B)$. Then taking $ \widetilde{g} = \omega^{-1}B $, we have $ \det(\widetilde{g}) = \omega^{-N}\det(B) = 1 $. Furthermore, since $ m $ divides both $ p^{r} $ and $ d $, we have $ \widetilde{g}^{p^{r}} = B^{p^{r}} = \I_{V} $ and $ F \circ \widetilde{g} = F \circ B = F $. Its projective image has order $p^r$, so $\widetilde g$ also has order $p^r$. Thus $\widetilde{g}$ is an $F$-lifting of $g$ in $\SL(V)$.
	
		For general $ \KK $, we have proved that there exists an $ F $-lifting $ \widetilde{g} $ of $ g $ in $ \SL(V\otimes\overline{\KK}) $, where $ \overline{\KK} $ is the algebraic closure of $ \KK $. Since $ p \nmid N $, Lemma
		\ref{lemma: liftability when gcd order N d=1} gives a lifting
		$C\in\GL(V)$ of $g$. Write $C=\mu\widetilde g$ with
		$\mu\in\overline{\KK}^\times$. Since both $\mu\widetilde g$ and $\widetilde g$ have order $p^r$, we have $\mu^{p^r}=1$. As $\det(\widetilde g)=1$, we also have
		\[
		\mu^N=\det(C)\in\KK^\times.
		\]
		Since $ p \nmid N $, we have $ \mu \in \KK^{\times}$, which implies that $ \widetilde{g} \in \SL(V) $.
	
	If there is another such lifting $ \overline{g} $, write
	$\widetilde g=\lambda\overline g$. Then
	$\lambda^{p^r}=\lambda^N=1$. Since $p\nmid N$, we have
	$\lambda=1$.
\end{proof}

\section{Proof of Theorem \ref{main theorem: F-liftability and Sylow subgroups}}
\label{section: proof of main theorem}
Let $G$ be a $p$-subgroup of $\PGL(V)$. For convenience, we call a finite subgroup $\overline{G} \subseteq \pi^{-1}(G)$ a \textbf{coprime cover} of $ G $ if $\pi|_{\overline{G}}\colon\overline{G} \rightarrow G$ is surjective and $ p \nmid \frac{|\overline{G}|}{|G|} $. 

\begin{lem}
	\label{lemma: reduction to coprime cover}
	Let $\KK$ be an arbitrary field, let $V$ be a finite-dimensional
	$\KK$-vector space, let $p$ be a prime, and let $G$ be a $p$-subgroup
	of $\PGL(V)$. Then $G$ is liftable if and only if it admits a coprime
	cover. Moreover, if $F\in S^d(V^*)$ is a nonzero homogeneous form and
	$G\subseteq\Lin(F)$, then $G$ is $F$-liftable if and only if it admits
	an $F$-preserving coprime cover.
\end{lem}

\begin{proof}
	Every lifting is a coprime cover, and every $F$-lifting is an
	$F$-preserving coprime cover. Conversely, let $\overline G$ be a
	coprime cover and put $K=\ker(\pi|_{\overline G})$. If
	$\widetilde G$ is a Sylow $p$-subgroup of $\overline G$, then
	$p\nmid|K|$ and
	\[
	|\widetilde G|=|G|,
	\qquad \widetilde G\cap K=\{\I_V\}.
	\]
	Thus $\pi|_{\widetilde G}\colon\widetilde G\to G$ is an isomorphism,
	so $\widetilde G$ is a lifting of $G$. If $\overline G$ is
	$F$-preserving, then so is $\widetilde G$, and hence it is an
	$F$-lifting of $G$.
\end{proof}

\begin{lem}
	\label{lemma: construction of finite cover for coprime cases}
		Let $\KK$ be an arbitrary field of characteristic zero, let $V$ be an $N$-dimensional $\KK$-vector space, and let $F\in S^d(V^*)$ be a nonsingular homogeneous form of degree $d\geq1$. Let $p$ be a prime with $p \nmid \gcd(N,d)$. Then any $p$-subgroup of $\Lin(F)$ admits an $ F $-preserving coprime cover.
		In particular, any $p$-subgroup of $\Lin(F)$ is $F$-liftable.
\end{lem}

\begin{proof}
Let $G$ be a $p$-subgroup of $\Lin(F)$. Suppose first that $p\nmid d$, and set
	\[
	\overline G
	=\{A\in\pi^{-1}(G)\mid F\circ A=F\}.
	\]
	This is an $F$-preserving subgroup of $\GL(V)$. Proposition
		\ref{proposition: F-liftability when order coprime to d} shows that
	$\pi|_{\overline G}$ is surjective, and
	\[
	\ker(\pi|_{\overline G})
	=\{\lambda\I_V\mid \lambda^d=1,\ \lambda\in\KK^\times\}.
	\]
	Thus the kernel is finite of order dividing $d$. Hence $\overline G$
	is finite and $|\overline G|/|G|$ is coprime to $p$, so $\overline G$
	is an $F$-preserving coprime cover of $G$.
	
	Now suppose that $p\mid d$. Then $p\nmid N$, and we set
	\[
	\overline G
	=\{A\in\pi^{-1}(G)\cap\SL(V)\mid F\circ A=F\}.
	\]
	By Proposition
	\ref{proposition: F-liftability when p mid d and p nmid N}, the map
	$\pi|_{\overline G}$ is surjective. Its kernel is
	\[
	\ker(\pi|_{\overline G})
	=\{\lambda\I_V\mid \lambda^d=\lambda^N=1,
	\ \lambda\in\KK^\times\}.
	\]
	This kernel is finite of order dividing $\gcd(N,d)$, which is coprime
	to $p$ because $p\nmid N$. It follows again that $\overline G$ is an
	$F$-preserving coprime cover of $G$.
	In either case, Lemma~\ref{lemma: reduction to coprime cover} shows
	that $G$ is $F$-liftable.
\end{proof}

We are now ready to prove the main Theorem~\ref{main theorem: F-liftability and Sylow subgroups}. 

\begin{proof}[Proof of Theorem~\ref{main theorem: F-liftability and Sylow subgroups}]
		If $G$ is $F$-liftable, then every subgroup of $G$ is
		$F$-liftable. In particular, $G$ has an $F$-liftable Sylow
		$p$-subgroup for every prime $p$ dividing $\gcd(|G|,N,d)$.
		
		Conversely, suppose that $G$ has an $F$-liftable Sylow
		$p$-subgroup for every prime $p$ dividing $\gcd(|G|,N,d)$. Let $p$ be
		any prime dividing $|G|$. If $p\mid\gcd(N,d)$, the required
		$F$-liftable Sylow $p$-subgroup is provided by the hypothesis. If
		$p\nmid\gcd(N,d)$, every $p$-subgroup of $\Lin(F)$ is $F$-liftable
		by Lemma~\ref{lemma: construction of finite cover for coprime cases}. Thus $G$ has an
		$F$-liftable Sylow $p$-subgroup for every prime $p$ dividing
		$|G|$. Theorem~\ref{theorem: F-liftability and Sylow subgroups}
		then implies that $G$ is $F$-liftable.
\end{proof}

\begin{rmk}
	\label{remark: recall [GALM] and [OY]}
	Over $\CC$, Theorem~\ref{main theorem: F-liftability and Sylow subgroups}
	recovers the "if part" of
	\cite[Theorem 3.5]{gonzalez2020lift} when $\gcd(N,d)=1$. If $d=p$
	is prime, it also strengthens
	\cite[Theorem 4.8]{oguiso2019quintic} by removing its additional condition (2).
\end{rmk}

	\bibliography{reference}
\end{document}